\documentclass[12pt,a4paper]{amsart}

\usepackage{hyperref} 
\usepackage{amsrefs}
\usepackage{amssymb}
\usepackage[all]{xy}

\newcommand{\calA}{\mathcal{A}}
\newcommand{\calD}{\mathcal{D}}
\newcommand{\calI}{\mathcal{I}}
\newcommand{\calK}{\mathcal{K}}
\newcommand{\calP}{\mathcal{P}}
\newcommand{\calT}{\mathcal{T}}
\newcommand{\calX}{\mathcal{X}}
\newcommand{\calY}{\mathcal{Y}}
\newcommand{\calZ}{\mathcal{Z}}

\newcommand{\bbF}{\mathbb{F}}
\newcommand{\bbN}{\mathbb{N}}
\newcommand{\bbZ}{\mathbb{Z}}

\newcommand{\frakR}{\mathfrak{R}}
\newcommand{\frakZ}{\mathfrak{Z}}

\newcommand{\bh}{\mathbf{h}}
\newcommand{\bv}{\mathbf{v}}

\let\Im=\undefined
\let\mod=\undefined
\DeclareMathOperator{\Id}{Id} %
\DeclareMathOperator{\Im}{Im} %
\DeclareMathOperator{\Hom}{Hom} %
\DeclareMathOperator{\Ker}{Ker} %
\DeclareMathOperator{\mod}{mod} %
\DeclareMathOperator{\nil}{nil} %
\DeclareMathOperator{\red}{red} %
\DeclareMathOperator{\chr}{char} %
\DeclareMathOperator{\proj}{proj} %
\DeclareMathOperator{\gldim}{gl.dim} %
\DeclareMathOperator{\hdim}{\mathbf{hdim}} %

\newcommand{\vertexD}[1]{\bullet \save*+!D{\scriptstyle #1} \restore}
\newcommand{\vertexL}[1]{\bullet \save*+!L{\scriptstyle #1} \restore}

\newcommand{\vertexU}[1]{\bullet \save*+!U{\scriptstyle #1} \restore}

\newcounter{claim}[section]
\renewcommand{\theclaim}{\thesection.\arabic{claim}}

\newtheorem{lemma}[claim]{Lemma}
\newtheorem{proposition}[claim]{Proposition}
\newtheorem{theorem}[claim]{Theorem}

\newtheorem*{theorems}{Theorem}

\title{The graded centers of derived discrete algebras}

\author{Grzegorz Bobi\'nski}

\address{Faculty of Mathematics and Computer Science \\ Nicolaus
Copernicus University \\ ul.~Chopina 12/18 \\ 87-100 Toru\'n \\
Poland}

\email{gregbob@mat.uni.torun.pl}

\date{}

\keywords{derived discrete algebras, the graded center of a
triangulated category}

\subjclass[2000]{16G20, 18E30}

\begin{document}

\begin{abstract}
We describe in the paper the graded centers of the derived
categories of the derived discrete algebras. In particular, we
prove that if $A$ is a derived discrete algebra, then the reduced
part of the graded center of the derived category of $A$ is
nontrivial if and only if $A$ has infinite global dimension.
Moreover, we show that the nilpotent part of the graded center is
controlled by the objects for which the Auslander--Reiten
translation coincides with a power of the suspension functor.
\end{abstract}

\maketitle

Throughout the paper $\bbF$ denotes a fixed algebraically closed
field. All considered categories are $\bbF$-categories and the
considered algebras and modules are finite dimensional over
$\bbF$. By $\bbZ$, $\bbN$, and $\bbN_+$, we denote the sets of
integers, nonnegative integers, and positive integers,
respectively. For $i, j \in \bbZ$, $[i, j] := \{ k \in \bbZ \mid i
\leq k \leq j \}$ (in particular, $[i, j] = \varnothing$ if $i >
j$). Moreover, $[i, \infty) := \{ k \in \bbZ \mid i \leq k \}$ and
$(-\infty, j] := \{ k \in \bbZ \mid k \leq j \}$.

\section*{Introduction}

For a triangulated category $\calT$ with the suspension functor
$\Sigma$ one defines the graded center $\frakZ (\calT)$ as the
graded ring which in degree $p \in \bbZ$ consists of all natural
transformations $\Id \to \Sigma^p$ which commute with $\Sigma$ up
to $(-1)^p$. In a series of papers~\cites{KessarLinckelmann2007,
Linckelmann2007, LinckelmannStancu2008} by Linckelmann, Kessar,
and Stancu, this notion has been proved useful in many situations,
when studying representations of finite groups. Moreover, Krause
and Ye~\cite{KrauseYe2008} studied the graded centers for some
classes of triangulated categories appearing in the representation
theory of finite dimensional algebras.

An important homological invariant of an algebra $A$ is the
derived category $\calD^b (A)$ of its module category. This
category has a structure of a triangulated category, thus it is
natural to study its graded center, which we denote by $\frakZ
(A)$. The algebras with the easiest to understand derived
categories are the derived discrete algebras described by
Vossieck~\cite{Vossieck2001}. Our aim in this paper is to
calculate the graded centers for the derived discrete algebras.

The precise description of the graded centers in the non-trivial
cases can be found in Section~\ref{sect_main}. We formulate here
only some consequences of this description. For a graded
commutative ring $R$ we denote by $R_{\nil}$ the ideal of the
nilpotent elements of $R$ and we put $R_{\red} := R / R_{\nil}$.
Moreover, given an algebra $A$ we denote by $\tau$ the
Auslander--Reiten translation in $\calD^b (A)$.

\begin{theorems}
Let $A$ be a derived discrete algebra and $R := \frakZ (A)$.
\begin{enumerate}

\item
$R_{\red} = \bbF$ if and only if $\gldim A < \infty$.

\item
$R_{\nil} \neq 0$ if and only if there exists an object $X$ in
$\calD^b (A)$ such that $\tau X \simeq \Sigma^p X$ for some $p \in
\bbZ$.

\end{enumerate}
\end{theorems}

The paper is organized as follows. In Section~\ref{sect_KrauseYe}
we recall the theorem of Krause and Ye stating that, when
calculating $\frakZ (A)$ for an algebra $A$, we may replace the
derived category $\calD^b (A)$ by the homotopy category $\calK^b
(\proj A)$ of perfect complexes over $A$. Next, in
Section~\ref{sect_discrete} we collect information about the
derived discrete algebras. As a consequence it follows that we may
concentrate in our calculations on the one-cycle gentle algebras
not satisfying the clock condition. In
Section~\ref{section_category} we describe $\calK^b (\proj A)$ for
a given one-cycle gentle algebra $A$ not satisfying the clock
condition of finite global dimension. This description is used in
Section~\ref{sect_calcul} in order to calculate the graded centers
for the one-cycle gentle algebras not satisfying the clock
condition of finite global dimension. Next, in
Section~\ref{section_infinite} we study the case of the one-cycle
gentle algebras not satisfying the clock condition of infinite
global dimension. Finally, in Section~\ref{sect_main} we summarize
our calculations.

For background on the representation theory of algebras (including
the language of quivers) we refer
to~\cite{AssemSimsonSkowronski2006}.

The article was written while the author was staying at the
University of Bielefeld as an Alexander von Humboldt Foundation
fellow. The author also acknowledges the support from the Research
Grant No.\ N N201 269135 of the Polish Ministry of Science and
High Education.

\section{The graded center of a triangulated
category}~\label{sect_KrauseYe}

Let $\calT$ be a triangulated category with the suspension functor
$\Sigma$. By the graded center $\frakZ (\calT)$ of $\calT$ we mean
the $\bbZ$-graded abelian group $\bigoplus_{p \in \bbZ} \frakZ_p
(\calT)$, where, for $p \in \bbZ$, $\frakZ_p (\calT)$ consists of
the natural transformations $\eta : \Id \to \Sigma^p$ such that
$\eta_{\Sigma X} = (-1)^p \cdot \Sigma \eta_X$ for each object $X$
of $\calT$. If $\eta' \in \frakZ_p (\calT)$ and $\eta'' \in
\frakZ_q (\calT)$ for $p, q \in \bbZ$, then we define the product
$\eta' \circ \eta''$ of $\eta'$ and $\eta''$ by $(\eta' \circ
\eta'')_X := \Sigma \eta'_X \circ \eta''_X$ for an object $X$ of
$\calT$. In this way we give $\frakZ (\calT)$ a structure of a
graded commutative ring, where a $\bbZ$-graded ring $R =
\bigoplus_{p \in \bbZ} R_p$ is called graded commutative if $r_1
\cdot r_2 = (-1)^{p \cdot q} \cdot r_2 \cdot r_1$ for all $r_1 \in
R_p$, $r_2 \in R_q$, $p, q \in \bbZ$. For completeness, one may
also study ``a commutative version'' of the graded center, i.e.\
the ring $\frakZ' (\calT) = \bigoplus_{p \in \bbZ} \frakZ_p'
(\calT)$, such that, for $p \in \bbZ$, $\frakZ_p' (\calT)$
consists of the natural transformations $\eta : \Id \to \Sigma^p$
such that $\eta_{\Sigma X} = \Sigma \eta_X$ for each object $X$ of
$\calT$. Obviously, $\frakZ' (\calT)$ is a commutative ring graded
by $\bbZ$.

Let $\calA$ be an additive category. By $\calK^b (\calA)$ we
denote the bounded homotopy category of $\calA$ defined as
follows. The objects of $\calK^b (\calA)$ are the differential
complexes $X = (X^p, d_X^p)$ of objects of $\calA$ such that $X^p
= 0$ for all but finite $p \in \bbZ$. If $X$ and $Y$ are objects
of $\calK^b (\calA)$, then $\Hom_{\calK^b (\calA)} (X, Y)$
consists of the equivalence classes of the morphisms $X \to Y$ of
complexes modulo the null-homotopic maps. Recall, that if $X$ and
$Y$ are complexes, then a morphism $f : X \to Y$ of complexes is
given by morphisms $f^p : X^p \to Y^p$, $p \in \bbZ$, in $\calA$
such that $f^{p + 1} \circ d_X^p = d_Y^p \circ f^p$ for each $p
\in \bbZ$. Moreover, if $X$ and $Y$ are complexes, then a morphism
$f : X \to Y$ is called null-homotopic, if there exist morphisms
$h^p : X^p \to Y^{p - 1}$, $p \in \bbZ$, in $\calA$, such that
$f^p = d_Y^{p - 1} \circ h^p + h^{p + 1} \circ d_X^p$. It is known
(see for example~\cite{Happel1988}*{I.3.2}) that $\calK^b (\calA)$
has a structure of a triangulated category with the suspension
functor $\Sigma$ given by the shift of complexes, i.e.\ if $p \in
\bbZ$, then $(\Sigma X)^p := X^{p + 1}$ and $d_{\Sigma X}^p := -
d_X^{p + 1}$ for an object $X$ of $\calK^b (\calA)$ and $(\Sigma
f)^p := f^{p + 1}$ for a morphism $f$ in $\calK^b (\calA)$.

Now assume that $\calA$ is an abelian category. Then for an object
$X$ of $\calK^b (\calA)$ and $p \in \bbZ$ we define the $p$-th
cohomology group $H^p (X)$ of $X$ by $H^p (X) := \Ker d_X^p / \Im
d_X^{p - 1}$. If $f \in \Hom_{\calK^b (\calA)} (X, Y)$ and $p \in
\bbZ$, then $f$ induces the map $H^p (f) : H^p (X) \to H^p (Y)$.
If $f \in \Hom_{\calK^b (\calA)} (X, Y)$, then $f$ is called a
quasi-isomorphism provided $H^p (f)$ is an isomorphism for each $p
\in \bbZ$. The derived category $\calD^b (\calA)$ of $\calA$ is by
definition the localization of $\calK^b (\calA)$ with respect to
the quasi-isomorphisms~\cite{Verdier1977}. It follows that
$\calD^b (\calA)$ has a structure of a triangulated category with
the suspension functor induced by the suspension functor in
$\calK^b (\calA)$.

The following theorem proved in~\cite{KrauseYe2008} will be a
useful tool in our calculations.

\begin{theorem}[Krause/Ye]
Let $\calA$ be an abelian category with enough projective objects.
Then $\frakZ (\calD^b (\calA))$ is positively graded. Moreover, if
$\calP$ is the full subcategory of $\calA$ consisting of the
projective objects, then $\frakZ (\calD^b (\calA)) \simeq \frakZ
(\calK^b (\calP))$.
\end{theorem}

We also have the following variant of the above theorem, whose
proof is almost identical.

\begin{theorem}
Let $\calA$ be an abelian category with enough projective objects.
Then $\frakZ' (\calD^b (\calA))$ is positively graded. Moreover,
if $\calP$ is the full subcategory of $\calA$ consisting of the
projective objects, then $\frakZ' (\calD^b (\calA)) \simeq \frakZ'
(\calK^b (\calP))$.
\end{theorem}

We will apply the above theorems in the situation when $\calA$ is
the category $\mod A$ of modules over an algebra $A$ and,
consequently, $\calP$ is the full subcategory $\proj A$ of $\mod
A$ formed by the projective modules. In the above situation we
will write just $\frakZ (A)$ ($\frakZ' (A)$) instead of $\frakZ
(\calD^b (\mod A))$ and $\frakZ (\calK^b (\proj A))$ ($\frakZ'
(\calD^b (\mod A))$ and $\frakZ' (\calK^b (\proj A))$,
respectively).

\section{Derived discrete algebras} \label{sect_discrete}

We begin this section with the definition of gentle algebras. Let
$(Q, \frakR)$ be a pair consisting of a finite connected quiver
(i.e.\ directed graph) $Q = (Q_0, Q_1)$, where $Q_0$ and $Q_1$ are
the sets of vertices and arrows in $Q$, respectively, and a set
$\frakR$ of paths in $Q$. We say that $(Q, \frakR)$ is a gentle
quiver if the following conditions are satisfied:
\begin{enumerate}

\item
for each $x \in Q_0$ there are at most two $\alpha \in Q_1$ such
that $s \alpha = x$ (i.e.\ $\alpha$ starts at $x$) and at most two
$\beta \in Q_1$ such that $t \beta = x$ (i.e.\ $\beta$ terminates
at $x$),

\item
$\frakR$ consists of paths of length $2$,

\item
for each $\alpha \in Q_1$ there is at most one $\beta \in Q_1$
such that $s \beta = t \alpha$ and $\beta \alpha \not \in \frakR$,
and at most one $\gamma \in Q_1$ such that $t \gamma = s \alpha$
and $\alpha \gamma \not \in \frakR$,

\item
for each $\alpha \in Q_1$ there is at most one $\beta \in Q_1$
such that $s \beta = t \alpha$ and $\beta \alpha \in \frakR$, and
at most one $\gamma \in Q_1$ such that $t \gamma = s \alpha$ and
$\alpha \gamma \in \frakR$.

\end{enumerate}
If, in addition, the number of arrows in $Q$ equals the number of
vertices in $Q$, then we say that $(Q, \frakR)$ is a one-cycle
gentle quiver.

Let $(Q, \frakR)$ be a one-cycle gentle quiver. We say that
$\alpha \in Q_1$ is a cycle arrow if the quiver $(Q_0, Q_1
\setminus \{ \alpha \})$ is connected. Let $Q_1'$ and $Q_1''$ be
the sets of clockwise and anti-clockwise oriented cycle arrows,
respectively (we leave to the reader to formulate the formal
definition of these notions). We say that $(Q, \frakR)$ satisfies
the clock condition if the number of $\alpha \beta \in \frakR$
such that $\alpha, \beta \in Q_1'$ equals the number of $\alpha
\beta \in \frakR$ such that $\alpha, \beta \in Q_1''$. Note that
this condition is obviously independent on the choice of
orientation.

An algebra $A$ is called gentle one-cycle not satisfying the clock
condition if $A \simeq \bbF Q / \langle \frakR \rangle$ for a
one-cycle gentle quiver $(Q, \frakR)$ not satisfying the clock
condition. Here, for a quiver $Q$ we denote by $\bbF Q$ the path
algebra of $Q$. Moreover, if $Q$ is a quiver and $\frakR$ is a set
of paths in $Q$, then $\langle \frakR \rangle$ denotes the ideal
in $\bbF Q$ generated by $\frakR$.

Let $A$ be an algebra. For a complex $X$ of $A$-modules we define
its cohomology dimension vector $\hdim X \in \bbN^{\bbZ}$ by
$(\hdim X)_p := \dim_{\bbF} H^p (X)$, $p \in \bbZ$. We say that
$A$ is derived discrete if for each $\bh \in \bbN^{\bbZ}$ the
indecomposable objects in $\calD^b (\mod A)$ with the cohomology
dimension vector $\bh$ form only a finite number of the
isomorphism classes.

Examples of derived discrete algebras are the hereditary algebras
of Dynkin type. Vossieck proved~\cite{Vossieck2001} that a
connected algebra $A$ is derived discrete if and only if either
$A$ is derived equivalent to a hereditary algebra of Dynkin type
(i.e.\ $\calD^b (\mod A)$ is equivalent as a triangulated category
to $\calD^b (\mod B)$ for a hereditary algebra $B$ of Dynkin type)
or $A$ is Morita equivalent to a one-cycle gentle algebra which
does not satisfy the clock condition. One easily verifies that
$\frakZ (A) = \bbF$ and $\frakZ' (A) = \bbF$ if $A$ is a
hereditary algebra of Dynkin type (see also~\cite{KrauseYe2008}).
Consequently, we concentrate in our paper on describing the graded
centers for the one-cycle gentle algebras which does not satisfy
the clock condition.

For $n \in \bbN_+$ and $m \in \bbN$ let $\Delta (n, m)$ be the
following quiver
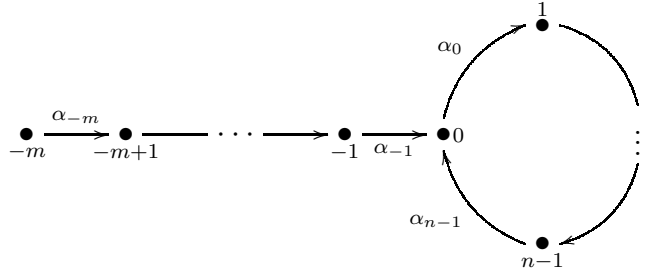
\[
\vcenter{\xymatrix{%
& & & & & \vertexD{1} \ar@{-}@/^/[]+R;[rd]+U &
\\
\vertexU{-m} \ar[r]^{\alpha_{-m}} & \vertexU{- m + 1} \ar@{-}[r] &
\cdots \ar[r] & \vertexU{-1} \ar[r]_{\alpha_{-1}} & \vertexL{0}
\ar@/^/[]+U;[ru]+L^{\alpha_0} & & \vdots \ar@/^/[]+D;[ld]+R
\\
& & & & & \vertexU{n - 1} \ar@/^/[]+L;[lu]+D^{\alpha_{n - 1}} &}}.
\]
Next, for $(r, n) \in \bbN_+^2$ such that $r \in [1, n]$ we put
\[
\frakR (r, n) := \{ \alpha_{i + 1} \alpha_i \mid i \in [n - r, n -
2] \} \cup \{ \alpha_0 \alpha_{n - 1} \}.
\]
Let $\Omega$ denote the set of triples $(r, n, m) \in \bbN^3$ such
that $n \in \bbN_+$ and $r \in [1, n]$. For $(r, n, m) \in \Omega$
we put $\Lambda (r, n, m) := \bbF \Delta (n, m) / \langle \frakR
(r, n) \rangle$. It is easy to check that $(\Delta (n, m), \frakR
(r, n))$ is a one-cycle gentle quiver not satisfying the clock
condition for each $(r, n, m) \in \Omega$. It was proved
in~\cite{BobinskiGeissSkowronski2004} that if $A$ is a one-cycle
gentle algebra not satisfying the clock condition, then there
exists $(r, n, m) \in \Omega$ such that $A$ and $\Lambda (r, n,
m)$ are derived equivalent. Consequently, it is sufficient to
describe the graded centers for the algebras of the above form.

We end this section with the following remark. Let $(Q, \frakR)$
be a one-cycle gentle quiver not satisfying the clock condition.
Then one can determine $(r, n, m) \in \Omega$ such that $\bbF Q /
\langle \frakR \rangle$ and $\Lambda (r, n, m)$ are derived
equivalent using the invariant introduced by Avella-Alaminos and
Geiss~\cite{AvellaAlaminosGeiss2008}.

\section{Category} \label{section_category}

Throughout this section we fix $(r, n, m) \in \Omega$ such that $r
< n$ and we put $\Lambda := \Lambda (r, n, m)$. We remark that the
condition $r < n$ implies that $\gldim \Lambda < \infty$. In this
section we describe a quiver $\Gamma$ and relations such that the
full subcategory of the indecomposable objects in $\calK^b (\proj
\Lambda)$ is equivalent to the path category of $\Gamma$ modulo
the given relations. This description follows from calculations
made in~\cite{BobinskiGeissSkowronski2004} (we also refer
to~\cites{BekkertMerklen2003, Bobinski2009}).

First, for $i \in [0, r - 1]$ we put $I_i := \bbZ^2$,
\begin{gather*}
I_i' := \{ (a, b) \in \bbZ^2 \mid a \leq b + \delta_{i, 0} \cdot m
\},
\\
\intertext{and} %
I_i'' := \{ (a, b) \in \bbZ^2 \mid a + \delta_{i, 0} \cdot n \leq
b \},
\end{gather*}
where $\delta_{x, y}$ is the Kronecker delta. Then the vertices of
$\Gamma$ are $X_v^{(i)}$ for $i \in [0, r - 1]$ and $v \in I_i'$,
$Y_v^{(i)}$ for $i \in [0, r - 1]$ and $v \in I_i''$, and
$Z_v^{(i)}$ for $i \in [0, r - 1]$ and $v \in I_i$.

Now we describe the arrows in $\Gamma$. Moreover, in order to be
able to describe the relations in a compact way we associate to
each arrow an additional data, which we call the degree of an
arrow.

First fix $i \in [0, r - 1]$ and $v = (a, b) \in I_i'$. We put
\begin{gather*}
\calI_v'^{(i)} := [a, b + \delta_{i, 0} \cdot m ] \times [b,
\infty), \qquad \calX_v^{(i)} := [a, b + \delta_{i, 0} \cdot m]
\times \bbZ,
\\
\intertext{and} %
\calX_v'^{(i)} := (-\infty, a + \delta_{i, r - 1} \cdot m] \times
[a, b + \delta_{i, 0} \cdot m].
\end{gather*}
For $u \in \calI_v'^{(i)}$, $u \neq v$, there is an arrow $f_{v,
u}'^{(i)} : X_v^{(i)} \to X_u^{(i)}$ of degree $0$. Next, for $u
\in \calX_v^{(i)}$ there is an arrow $g_{v, u}'^{(i)} : X_v^{(i)}
\to Z_u^{(i)}$ of degree $1$. Finally, for $u \in \calX_v'^{(i)}$
there is an arrow $e_{v, u}'^{(i)} : X_v^{(i)} \to X_u^{(i + 1)}$
of degree $2$, where we always change the upper index modulo $r$.

Now fix $i \in [0, r - 1]$ and $v := (a, b) \in I_i''$. We put
\begin{gather*}
\calI_v''^{(i)} := [a, b - \delta_{i, 0} \cdot n] \times [b,
\infty), \qquad \calY_v^{(i)} := \bbZ \times [a, b - \delta_{i, 0}
\cdot n],
\\
\intertext{and} %
\calY_v''^{(i)} := (-\infty, a - \delta_{i, r - 1} \cdot n] \times
[a, b - \delta_{i, 0} \cdot n].
\end{gather*}
For $u \in \calI_v'^{(i)}$, $u \neq v$, there is an arrow $f_{v,
u}''^{(i)} : Y_v^{(i)} \to Y_u^{(i)}$ of degree $0$. Next, for $u
\in \calY_v^{(i)}$ there is an arrow $g_{v, u}''^{(i)} : Y_v^{(i)}
\to Z_u^{(i)}$ of degree $1$. Finally, for $u \in \calY_v''^{(i)}$
there is an arrow $e_{v, u}''^{(i)} : Y_v^{(i)} \to Y_u^{(i + 1)}$
of degree $2$.

Finally fix $i \in [0, r - 1]$ and $v := (a, b) \in I_i$. We put
\begin{gather*}
\calI_v^{(i)} := [a, \infty) \times [b, \infty),
\\
\calZ_v'^{(i)} := (-\infty, a + \delta_{i, r - 1} \cdot m] \times
[a, \infty),
\\
\calZ_v''^{(i)} := (-\infty, b - \delta_{i, r - 1} \cdot n] \times
[b, \infty),
\\
\intertext{and} %
\calZ_v^{(i)} := (-\infty, a + \delta_{i, r - 1} \cdot m] \times
(\infty, b - \delta_{i, r - 1} \cdot n].
\end{gather*}
For $u \in \calI_v^{(i)}$, $u \neq v$, there is an arrow $f_{v,
u}^{(i)} : Z_v^{(i)} \to Z_u^{(i)}$ of degree $0$. Next, for $u
\in \calZ_v'^{(i)}$ there is an arrow $h_{v, u}'^{(i)} : Z_v^{(i)}
\to X_u^{(i + 1)}$ of degree $1$. Similarly, for $u \in
\calZ_v''^{(i)}$ there is an arrow $h_{v, u}''^{(i)} : Z_v^{(i)}
\to Y_u^{(i + 1)}$ of degree $1$. Finally, for $u \in
\calZ_v^{(i)}$ there is an arrow $e_{v, u}''^{(i)} : Z_v^{(i)} \to
Z_u^{(i + 1)}$ of degree $2$.

Now we describe the relations. Let $f : X \to Y$ and $g : Y \to Z$
be arrows of degree $p$ and $q$, respectively. If there is an
arrow $h : X \to Z$ of degree $p + q$, then we have the relation
$g f = h$, otherwise we have the relation $g f = 0$ (we note that
some of the relations are described directly in
Section~\ref{section_infinite}).

We also introduce the following notation
\begin{gather*}
f_{v, v}'^{(i)} := \Id_{X_v^{(i)}}, \; v \in I_i', \qquad f_{v,
v}''^{(i)} := \Id_{Y_v^{(i)}}, \; v \in I_i'',
\\
\intertext{and} %
f_{v, v}^{(i)} := \Id_{Z_v^{(i)}}, \; v \in I_i, \,
\end{gather*}
for $i \in [0, r - 1]$.

Now we describe the shift $\Sigma$. First, for $i \in [0, r - 1]$
we put
\begin{gather*}
\bv_i' := (1 + \delta_{i, r - 1} \cdot m, 1 + \delta_{i, 0} \cdot
m), \qquad \bv_i'' := (1 - \delta_{i, r - 1} \cdot n, 1 -
\delta_{i, 0} \cdot n),
\\
\intertext{and} %
\bv_i := (1 + \delta_{i, r - 1} \cdot m, 1 - \delta_{i, r - 1}
\cdot n).
\end{gather*}
If $i \in [0, r - 1]$, then $\Sigma X_v^{(i)} = X_{v + \bv_i'}^{(i
+ 1)}$ for each $v \in I_i'$, $\Sigma Y_v^{(i)} = Y_{v +
\bv_i''}^{(i + 1)}$ for each $v \in I_i''$, and $\Sigma Z_v^{(i)}
= Z_{v + \bv_i}^{(i + 1)}$ for each $v \in I_i''$. We leave it to
the reader to write the obvious formulas for the action of the
shift on the morphisms.

Finally, we describe the action of the Auslander--Reiten
translation $\tau$. Namely, if $i \in [0, r - 1]$, then $\tau
X_v^{(i)} = X_{v - (1, 1)}^{(i)}$ for each $v \in I_i'$, $\tau
Y_v^{(i)} = Y_{v - (1, 1) }^{(i)}$ for each $v \in I_i''$, and
$\tau Z_v^{(i)} = Z_{v - (1, 1)}^{(i)}$ for each $v \in I_i''$.
Observe that it follows by direct calculations that there exists
an indecomposable object $X$ in $\calK^b (\proj \Lambda)$ such
that $\tau X = \Sigma^p X$ for some $p \in \bbZ$ if and only if
either $r = n - 1$ or $r = 1$ and $m = 0$.

\section{Calculations} \label{sect_calcul}

Throughout this section we fix $(r, n, m) \in \Omega$ such that $r
< n$ and we put $\Lambda := \Lambda (r, n, m)$. In this section we
calculate $\frakZ (\Lambda)$ and $\frakZ (\Lambda')$.

First we describe the homomorphism spaces between the
indecomposable objects in $\calK^b (\proj \Lambda)$ and their
shifts.

\begin{lemma} \label{lemma_Zp}
Let $i \in [0, r - 1]$ and $v \in I_i$. If $p \in \bbN$, then
\[
\Hom_{\calK^b (\proj \Lambda)} (Z_v^{(i)}, \Sigma^p Z_v^{(i)}) =
\begin{cases}
\bbF \cdot f_{v, v}^{(i)} & \text{$p = 0$},
\\
0 & \text{$p > 0$}.
\end{cases}
\]
\end{lemma}

\begin{proof}
Put $Z := Z_v^{(i)}$. The claim is obvious if neither $p - 1$ nor
$p$ is divisible by $r$.

First assume that $r > 1$ and $p = q r + 1$ for some $q \in \bbN$.
Then $\Sigma^p Z = Z_{v + q (r + m, r - n) + \bv_i}^{(i + 1)}$ and
one easily verifies that $v + q (r + m, r - n) + \bv_i \not \in
\calZ_v^{(i)}$.

Now assume that $p = q r$ for some $q \in \bbN$. In this case
$\Sigma^p Z = Z_{v + q (r + m, r - n)}^{(i)}$. We leave it to the
reader to verify that $v + q (r + m, r - n) \in \calI_v^{(i)}$ if
and only if $q = 0$. Finally, one also shows that $v + q (r + m, r
- n) \not \in \calZ_v^{(0)}$ provided $r = 1$.
\end{proof}

\begin{lemma} \label{lemma_Xp}
Let $i \in [0, r - 1]$ and $v = (a, b) \in I_i'$. If $p \in
\bbN_+$, then
\begin{multline*}
\Hom_{\calK^b (\proj \Lambda)} (X_v^{(i)}, \Sigma^p X_v^{(i)}) =
\\
\begin{cases}
\bbF \cdot f_{v, v + \frac{p}{r} (r + m, r + m)}'^{(i)} & \text{$r
\mid p$ and $\frac{p}{r} \leq \frac{b + \delta_{i, 0} \cdot m -
a}{r + m}$},
\\
0 & \text{otherwise}.
\end{cases}
\end{multline*}
Moreover,
\[
\Hom_{\calK^b (\proj \Lambda)} (X_v^{(i)}, X_v^{(i)}) =
\begin{cases}
\bbF \cdot f_{v, v}'^{(i)} + \bbF \cdot e_{v, v}'^{(i)} & \text{$r
= 1$ and $a \leq b$}.
\\
\bbF \cdot f_{v, v}'^{(i)} & \text{otherwise}.
\\
\end{cases}
\]
\end{lemma}

\begin{proof}
The method of the proof is analogous to that of the proof of the
previous lemma, hence we leave it to the reader.
\end{proof}

\begin{lemma} \label{lemma_Yp}
Let $i \in [0, r - 1]$ and $v = (a, b) \in I_i''$. If $p \in
\bbN$, then
\begin{multline*}
\Hom_{\calK^b (\proj \Lambda)} (Y_v^{(i)}, \Sigma^p Y_v^{(i)}) =
\\
\begin{cases}
\bbF \cdot f_{v, v}''^{(i)} & \text{$p = 0$},
\\
\bbF \cdot e_{v, v + \frac{p - 1}{r} (r - n, r - n) +
\bv_i''}''^{(i)} & \text{$r \mid p - 1$ and}
\\
& \quad \text{$\frac{1}{n - r} \leq \frac{p - 1}{r} \leq \frac{b +
1 - a - \delta_{i, 0} \cdot n}{n - r}$},
\\
0 & \text{otherwise}.
\end{cases}
\end{multline*}
\end{lemma}

\begin{proof}
Similar as above.
\end{proof}

Assume that $r = n - 1$ and fix $q \in \bbN$. Observe that in this
case $\tau Y_v^{(i)} = \Sigma^r Y_v^{(i)}$ for each $i \in [0, r -
1]$ and $v \in I_i''$. Moreover, for each $i \in [0, r - 1]$ and
$v = (a, b) \in I_i''$ such that $b - a = q + \delta_{i, 0} \cdot
n$, there exists unique $p \in \bbZ$ such that $Y_v^{(i)} =
\Sigma^p Y_{(0, n + q)}^{(0)}$. We put $\varepsilon (i, v) :=
(-1)^{n \cdot p}$ in the above situation. We define the natural
transformation $\eta'^{(q)} : \Id \to \Sigma^n$ by setting
$\eta_X'^{(q)} := \varepsilon (i, v) \cdot e_{v, v + \bv_i'' - (1,
1)}''^{(i)}$ if $X = Y_v^{(i)}$ for $i \in [0, r - 1]$ and $v =
(a, b) \in I_i''$ such that $b - a = q + \delta_{i, 0} \cdot n$,
and $\eta_X'^{(q)} := 0$ if $X$ is an indecomposable object of
$\calK^b (\proj \Lambda)$ not isomorphic to $Y_v^{(i)}$ for some
$i \in [0, r - 1]$ and $v = (a, b) \in I_i''$ such that $b - a = q
+ \delta_{i, 0} \cdot n$. One easily verifies that $\eta'^{(q)}
\in \frakZ_n (\Lambda)$.

\begin{proposition} \label{prop_fdp}
Let $p \in \bbN_+$. Then
\[
\frakZ_p (\Lambda) =
\begin{cases}
\prod_{q \in \bbN} \bbF \cdot \eta'^{(q)} & \text{$(r, p) = (n -
1, n)$},
\\
0 & \text{otherwise}.
\end{cases}
\]
\end{proposition}

\begin{proof}
Fix $\eta \in \frakZ_p (\Lambda)$. Lemma~\ref{lemma_Zp} implies
that $\eta_{Z_v^{(i)}} = 0$ for each $v \in I_i$ and $i \in [0, r
- 1]$.

Now we show that $\eta_X = 0$ if $X = X_v^{(i)}$ for $v = (a, b)
\in I_i'$ and $i \in [0, r - 1]$. It follows from from
Lemma~\ref{lemma_Xp} that we may assume that $r \mid p$ and
$\frac{p}{r} \leq \frac{b + \delta_{i, 0} m - a}{r + m}$. In this
case $\Sigma^p X = X_u^{(i)}$ and $\eta_X = \lambda \cdot f_{v,
u}'^{(i)}$ for some $\lambda \in \bbF$, where $u := v +
\frac{p}{r} (r + m, r + m)$. Let $f := g_{v, v'}'^{(i)}$, where
$v' := (a, 0)$. Then $\Sigma^p f =  g_{u, u'}'^{(i)}$, where $u'
:= v' + \frac{p}{r} (r + m, r - n)$. Observe that $\Sigma^p f
\circ \eta_X = \lambda \cdot g_{v, u'}'^{(i)}$. On the other hand
$\eta_{Z_{v'}^{(i)}} \circ f = 0$, as we have already proved that
$\eta_{Z_{v'}^{(i)}} = 0$. Since $\Sigma^p f \circ \eta_X =
\eta_{Z_{v'}^{(i)}} \circ f$, it follows that $\lambda = 0$, and
hence $\eta_X = 0$.

Now assume that $(r, p) \neq (n - 1, n)$ and let $Y := Y_v^{(i)}$
for some $v = (a, b) \in I_i''$ and $i \in [0, r - 1]$. We prove
by induction on $b - a - \delta_{i, 0} \cdot n$ that $\eta_Y = 0$.
If $b = a + \delta_{i, 0} \cdot n$, then the claim follows from
Lemma~\ref{lemma_Yp} (due to the assumption $(r, p) \neq (n - 1,
n)$). Now assume that $b > a + \delta_{i, 0} \cdot n$.
Lemma~\ref{lemma_Yp} implies that we may assume that $r$ divides
$p - 1$ and $\frac{1}{n - r} \leq \frac{p - 1}{r} \leq \frac{b + 1
- a - \delta_{i, 0} \cdot n}{n - r}$. In this case $\Sigma^p Y =
Y_u^{(i + 1)}$ and $\eta_Y = \lambda \cdot e_{v, u}''^{(i)}$ for
some $\lambda \in \bbF$, where $u := v + \frac{p - 1}{r} (r - n, r
- n) + \bv_i''$. Let $v' := (a, b - 1)$ and $Y' := Y_{v'}^{(i)}$.
By the induction hypothesis $\eta_{Y'} = 0$, thus $\eta_Y \circ
f_{v', v}''^{(i)} = \Sigma^p f_{v', v}''^{(i)} \circ \eta_{Y'} =
0$. On the other hand, using once more the assumption $(r, p) \neq
(n - 1, n)$, we get $u \in \calY_{v'}''^{(i)}$, hence $\eta_Y
\circ f_{v', v}''^{(i)} = \lambda \cdot e_{v', u}''^{(i)}$ and the
claim follows.

Finally, assume that $(r, p) = (n - 1, n)$. In this case for each
$i \in [0, r - 1]$ and $v \in I_i''$ there exists $\lambda_{i, v}
\in \bbF$ such that $\eta_{Y_v^{(i)}} = \lambda_{i, v} \cdot e_{v,
v + \bv_i'' - (1, 1)}''^{(i)}$. Since $\eta$ commutes up to sign
with $\Sigma$, $\lambda_{i, v} = \varepsilon (i, v) \cdot
\lambda_{0, (0, n + b - a - \delta_{i, 0} \cdot n)}$ for each $i
\in [0, r - 1]$ and $v = (a, b) \in I_i''$, hence the claim
follows.
\end{proof}

Again assume that $r = n - 1$ and fix $q \in \bbN$. Similarly as
before we define the natural transformation $\eta''^{(q)} : \Id
\to \Sigma^n$ by $\eta_X''^{(q)} := e_{v, v + \bv_i'' - (1,
1)}''^{(i)}$ if $X = Y_v^{(i)}$ for $i \in [0, r - 1]$ and $v =
(a, b) \in I_i''$ such that $b - a = q + \delta_{i, 0} \cdot n$,
and $\eta_X'^{(q)} := 0$ if $X$ is an indecomposable object of
$\calK^b (\proj \Lambda)$ not isomorphic to $Y_v^{(i)}$ for some
$i \in [0, r - 1]$ and $v = (a, b) \in I_i''$ such that $b - a = q
+ \delta_{i, 0} \cdot n$. One easily verifies that $\eta''^{(q)}
\in \frakZ_n' (\Lambda)$.

We have the following variant of the above proposition, which is
proved analogously.

\addtocounter{claim}{-1}
\renewcommand{\theclaim}{\thesection.\arabic{claim}'}
\begin{proposition}
\pushQED{\qed} %
Let $p \in \bbN_+$. Then
\[
\frakZ_p (\Lambda) =
\begin{cases}
\prod_{q \in \bbN} \bbF \cdot \eta''^{(q)} & \text{$(r, p) = (n -
1, n)$},
\\
0 & \text{otherwise}.
\end{cases}
\qedhere
\]
\popQED %
\end{proposition}
\renewcommand{\theclaim}{\thesection.\arabic{claim}}

Our next aim is to calculate $\frakZ_0 (\Lambda) = \frakZ_0'
(\Lambda)$. Let $\Id$ denote the natural transformation $\Id \to
\Id$ which associates to an object $X$ of $\calK^b (\proj
\Lambda)$ the identity map. Obviously $\Id \in \frakZ_0
(\Lambda)$.

Assume that $r = 1$ and $m = 0$, and fix $q \in \bbN$. Observe
that in this case $\tau X_v^{(0)} = \Sigma^{-1} X_v^{(0)}$ for
each $v \in I_0'$. We define the natural transformation
$\eta^{(q)} : \Id \to \Id$ by setting $\eta_X^{(q)} := e_{v,
v}'^{(0)}$ if $X = X_v^{(0)}$ for $v = (a, b) \in I_0'$ such that
$b - a = q$, and $\eta_X^{(q)} := 0$ if $X$ is an indecomposable
object of $\calK^b (\proj \Lambda)$ not isomorphic to $X_v^{(0)}$
for some $v = (a, b) \in I_0'$ such that $b - a = q$. One easily
verifies that $\eta^{(q)} \in \frakZ_0 (\Lambda)$.

\begin{proposition} \label{prop_zzero}
We have
\[
\frakZ_0 (\Lambda) =
\begin{cases}
\bbF \cdot \Id \oplus \prod_{q \in \bbN} \bbF \cdot \eta^{(q)} &
\text{$r = 1$ and $m = 0$},
\\
\bbF \cdot \Id & \text{otherwise}.
\end{cases}
\]
\end{proposition}

\begin{proof}
Fix $\eta \in \frakZ_0 (\Lambda)$. Let $\lambda \in \bbF$ be such
that $\eta_{Z_{(0, 0)}^{(0)}} = \lambda \cdot f_{(0, 0), (0,
0)}^{(0)}$.

First we show by induction on $i$ that for each $i \in [0, r - 1]$
there exists $v \in I_i$ such that $\eta_{Z_v^{(i)}} = \lambda
\cdot f_{v, v}^{(i)}$. This claim is obvious for $i = 0$, thus
assume that $i > 0$. Fix $u \in I_{i - 1}$ such that
$\eta_{Z_{u}^{(i - 1)}} = \lambda \cdot f_{u, u}^{(i - 1)}$. Put
$v := u + \bv_{i - 1} - (1, 1)$. Lemma~\ref{lemma_Zp} implies that
$\eta_{Z_v^{(i)}} = \mu \cdot f_{v, v}^{(i)}$ for some $\mu \in
\bbF$. Moreover, $\mu \cdot e_{u, v}^{(i - 1)} = \eta_{Z_v^{(i)}}
\circ e_{u, v}^{(i - 1)} = e_{u, v}^{(i - 1)} \circ
\eta_{Z_{u}^{(i - 1)}} = \lambda \cdot e_{u, v}^{(i - 1)}$, hence
$\mu = \lambda$.

Next we show that $\eta_{Z_v^{(i)}} = \lambda \cdot f_{v,
v}^{(i)}$ for each $i \in [0, r - 1]$ and $v \in I_i$. We proceed
in a similar way as above in the following steps. First, we fix $u
\in I_i$ such that $\eta_{Z_u^{(i)}} = \lambda \cdot f_{u,
u}^{(i)}$. Next we show the claim for all $v \in \calI_u^{(i)}$
(using $f_{u, v}^{(i)}$) and finally we prove it for arbitrary $v
\in I_i$ (using $f_{v, v'}^{(i)}$ for some $v' \in \calI_v^{(i)}
\cap \calI_u^{(i)}$).

Further we apply the same method and Lemma~\ref{lemma_Yp} in order
to show that $\eta_{Y_v^{(i)}} = \lambda \cdot f_{v, v}''^{(i)}$
for each $i \in [0, r - 1]$ and $v \in I_i''$. In this case we use
$g_{v, u}''^{(i)}$ for some $u \in \calY_v^{(i)}$. Moreover, we
use Lemma~\ref{lemma_Xp} in order to prove similarly that
$\eta_{X_v^{(i)}} = \lambda \cdot f_{v, v}'^{(i)}$ for each $i \in
[0, r - 1]$ and $v \in I_i'$ provided $r > 1$.

Finally assume that $r = 1$. The analogous arguments to those
presented above and Lemma~\ref{lemma_Xp} imply that for each $v
\in I_0'$ there exists $\lambda_v \in \bbF$ such that
$\eta_{X_v^{(0)}} = \lambda \cdot f_{v, v}'^{(0)} + \lambda_v
\cdot e_{v, v}'^{(0)}$. Observe that for each $v = (a, b) \in
I_0''$ there exists unique $p \in \bbZ$ such that $X_v^{(i)} =
\Sigma^p X_{(0, b - a)}^{(0)}$. Consequently, it follows that
$\lambda_v = \lambda_{(0, b - a)}$ for each $v = (a, b) \in I_0'$,
since $\eta$ and $\Sigma$. Finally, we prove by induction on $a
\in \bbN$ that $\lambda_{(0, a)} = 0$ if $m > 0$, proceeding
similarly as we did in the proof of Proposition~\ref{prop_fdp}.
\end{proof}

Finally, we describe the multiplication in $\frakZ (\Lambda)$ and
$\frakZ' (\Lambda)$.

\begin{proposition}
Let $q_1, q_2 \in \bbN$.
\begin{enumerate}

\item
If $r = n - 1$, then $\eta'^{(q_1)} \cdot \eta'^{(q_2)} = 0 =
\eta''^{(q_1)} \cdot \eta''^{(q_2)}$.

\item
If $r = 1$ and $m = 0$, then $\eta^{(q_1)} \cdot \eta^{(q_2)} =
0$.

\item
If $r = 1$, $n = 2$, and $m = 0$, then $\eta^{(q_1)} \cdot
\eta'^{(q_2)} = 0 = \eta^{(q_1)} \cdot \eta''^{(q_2)}$.

\end{enumerate}
\end{proposition}

\begin{proof}
Direct calculations.
\end{proof}

\section{Infinite global dimension} \label{section_infinite}

Throughout this section we fix $(n, m) \in \bbN^2$ such that $n >
0$ and we put $\Lambda := \Lambda (n, n, m)$. In this case $\gldim
\Lambda = \infty$.

First we describe the full subcategory of $\calK^b (\proj
\Lambda)$ formed by the indecomposable objects. Thoroughly
speaking it is the full subcategory of the category described in
Section~\ref{section_category} given by the $X$-vertices. More
precisely, we have the following quiver with relations whose path
category is equivalent to the full subcategory of $\calK^b (\proj
\Lambda)$ formed by the indecomposable objects. The vertices of
this quiver are $X_v^{(i)}$ for $i \in [0, n - 1]$ and $v \in
I_i'$, where $I_i'$ is defined as in
Section~\ref{section_category}, i.e.\
\[
I_i' := \{ (a, b) \in \bbZ^2 \mid a \leq b + \delta_{i, 0} \cdot m
\}.
\]
Next, for $i \in [0, n - 1]$ and $v = (a, b) \in I_i'$ we define
$\calI_v'^{(i)}$ and $\calX_v'^{(i)}$ by
\begin{gather*}
\calI_v'^{(i)} := [a, b + \delta_{i, 0} \cdot m ] \times [b,
\infty)
\\
\intertext{and} %
\calX_v'^{(i)} := (-\infty, a + \delta_{i, n - 1} \cdot m] \times
[a, b + \delta_{i, 0} \cdot m].
\end{gather*}
Then for each $i \in [0, n - 1]$, $v \in I_i'^{(i)}$, and $u \in
\calI_v'^{(i)}$, $u \neq v$, we have an arrow $f_{v, u}'^{(i)} :
X_v^{(i)} \to X_u^{(i)}$, and for each $i \in [0, n - 1]$, $v \in
I_i'^{(i)}$, and $u \in \calX_v'^{(i)}$, we have an arrow $e_{v,
u}'^{(i)} : X_v^{(i)} \to X_u^{(i + 1)}$. Moreover, we put $f_{v,
v}'^{(i)} := \Id_{X_v^{(i)}}$ for each $i \in [0, n - 1]$ and $v
\in I_i'$. Finally, we have the following relations:
\[
f_{u, w}'^{(i)} \circ f_{v, u}'^{(i)} =
\begin{cases}
f_{v, w}'^{(i)} & w \in \calI_v'^{(i)},
\\
0 & \text{otherwise},
\end{cases}
\]
for $i \in [0, n - 1]$, $v \in I_i'$, $u \in \calI_v'^{(i)}$, and
$w \in \calI_u'^{(i)}$,
\[
e_{u, w}'^{(i)} \circ f_{v, u}'^{(i)} =
\begin{cases}
e_{v, w}'^{(i)} & w \in \calX_v'^{(i)},
\\
0 & \text{otherwise},
\end{cases}
\]
for $i \in [0, n - 1]$, $v \in I_i'$, $u \in \calI_v'^{(i)}$, and
$w \in \calX_u'^{(i)}$,
\[
f_{u, w}'^{(i + 1)} \circ e_{v, u}'^{(i)} =
\begin{cases}
e_{v, w}'^{(i)} & w \in \calX_v'^{(i)},
\\
0 & \text{otherwise},
\end{cases}
\]
for $i \in [0, n - 1]$, $v \in I_i'$, $u \in \calX_v'^{(i)}$, and
$w \in \calI_u'^{(i + 1)}$, and
\[
e_{u, w}'^{(i + 1)} \circ e_{v, u}'^{(i)} = 0
\]
for $i \in [0, n - 1]$, $v \in I_i'$, $u \in \calX_v'^{(i)}$, and
$w \in \calX_u'^{(i + 1)}$.

The descriptions of $\Sigma$ and $\tau$ are also analogous to
these given in Section~\ref{section_category}. Namely, if $i \in
[0, n - 1]$, then $\Sigma X_v^{(i)} = X_{v + \bv_i'}^{(i + 1)}$
and $\tau X_v^{(i)} = X_{v - (1, 1)}^{(i)}$ for each $v \in I_i'$,
where $\bv_i' := (1 + \delta_{i, n - 1} \cdot m, 1 + \delta_{i, 0}
\cdot m)$. Consequently, there exists an indecomposable object $X$
in $\calK^b (\proj \Lambda)$ such that $\tau X = \Sigma^p X$ for
some $p \in \bbZ$ if and only if $n = 1$ and $m = 0$.

First we observe that Lemma~\ref{lemma_Xp} is valid in this case
without any changes. Namely, we have the following.

\begin{lemma}
Let $i \in [0, n - 1]$ and $v = (a, b) \in I_i'$. If $p \in
\bbN_+$, then
\begin{multline*}
\Hom_{\calK^b (\proj \Lambda)} (X_v^{(i)}, \Sigma^p X_v^{(i)}) =
\\
\begin{cases}
\bbF \cdot f_{v, v + \frac{p}{n} (n + m, n + m)}'^{(i)} & \text{$n
\mid p$ and $\frac{p}{n} \leq \frac{b + \delta_{i, 0} m - a}{n +
m}$},
\\
0 & \text{otherwise}.
\end{cases}
\end{multline*}
Moreover,
\[
\Hom_{\calK^b (\proj \Lambda)} (X_v^{(i)}, X_v^{(i)}) =
\begin{cases}
\bbF \cdot f_{v, v}'^{(i)} + \bbF \cdot e_{v, v}'^{(i)} & \text{$n
= 1$ and $a \leq b$}.
\\
\bbF \cdot f_{v, v}'^{(i)} & \text{otherwise}.
\\
\end{cases}
\]
\end{lemma}

The description of $\frakZ_0 (\Lambda) = \frakZ_0' (\Lambda)$ does
not differ either. Namely, let $\Id$ denote the natural
transformation $\Id \to \Id$ which associates to an object $X$ in
$\calK^b (\proj \Lambda)$ the identity map. Moreover, if $n = 1$
and $m = 0$, then for $q \in \bbN$ we define the natural
transformation $\eta^{(q)} : \Id \to \Id$ by setting $\eta_X^{(q)}
:= e_{v, v}'^{(0)}$ if $X = X_v^{(0)}$ for $v = (a, b) \in I_0'$
such that $b - a = q$, and $\eta_X^{(q)} := 0$ if $X$ is an
indecomposable object of $\calK^b (\proj \Lambda)$ not isomorphic
to $X_v^{(0)}$ for some $v = (a, b) \in I_i'$ such that $b - a =
q$.

The proof of the following fact is obtained by adapting the
arguments from the proof of Proposition~\ref{prop_zzero} to the
considered case.

\begin{proposition}
We have
\[
\frakZ_0 (\Lambda) =
\begin{cases}
\bbF \cdot \Id \oplus \prod_{q \in \bbN} \bbF \cdot \eta^{(q)} &
\text{$n = 1$},
\\
\bbF \cdot \Id & \text{otherwise}.
\end{cases}
\]
\end{proposition}

The situation differs in positive degrees.

We define $\eta : \Id \to \Sigma^n$ in the following way. We put
$\eta_X := f_{v, v + (n + m, n + m)}'^{(i)}$ if $X = X_v^{(i)}$
for $i \in [0, n - 1]$ and $v = (a, b) \in I_i'$ such that $n + m
\leq b + \delta_{i, 0} \cdot m - a$, and $\eta_X := 0$ if $X$ is
an indecomposable object of $\calK^b (\proj \Lambda)$ not
isomorphic to $X_v^{(i)}$ for some $i \in [0, n - 1]$ and $v = (a,
b) \in I_i'$ such that $n + m \leq b + \delta_{i, 0} \cdot m - a$.
Observe that $\eta^p \neq 0$ for each $p \in \bbN_+$. Moreover,
$\eta^p \in \frakZ_{p \cdot n}' (\Lambda)$ for each $p \in
\bbN_+$. More precisely, $\eta_X^p = f_{v, v + p \cdot (n + m, n +
m)}$ if $X = X_v^{(i)}$ for $i \in [0, n - 1]$ and $v = (a, b) \in
I_i'$ such that $p \cdot (n + m) \leq b + \delta_{i, 0} \cdot m -
a$, and $\eta_X := 0$ if $X$ is an indecomposable object of
$\calK^b (\proj \Lambda)$ not isomorphic to $X_v^{(i)}$ for some
$i \in [0, n - 1]$ and $v = (a, b) \in I_i'$ such that $p \cdot (n
+ m) \leq b + \delta_{i, 0} \cdot m - a$. Finally, if $p \in
\bbN_+$, then $\eta^p \in \frakZ_{p \cdot n} (\Lambda)$ if and
only if either $2 \mid p \cdot n$ or $\chr \bbF = 2$.

We have the following.

\begin{proposition}
Let $p \in \bbN_+$. Then
\[
\frakZ_p' (\Lambda) =
\begin{cases}
\bbF \cdot \eta^{\frac{p}{n}} & \text{$n \mid p$},
\\
0 & \text{otherwise}.
\end{cases}
\]
Moreover,
\[
\frakZ_p (\Lambda) =
\begin{cases}
\frakZ_p' (\Lambda) & \text{either $2 \mid p$ or $\chr \bbF = 2$},
\\
0 & \text{otherwise}.
\end{cases}
\]
\end{proposition}

\begin{proof}
The claim follows by arguments similar to those used before, hence
we omit it (compare also the proof
of~\cite{KrauseYe2008}*{Lemma~5.3}).
\end{proof}

We finish this section with the following.

\begin{proposition}
Let $q, q_1, q_2 \in \bbN$. If $n = 1$, then
\[
\eta \cdot \eta^{(q)} = 0 = \eta^{(q_1)} \cdot \eta^{(q_2)}.
\]
\end{proposition}

\begin{proof}
Direct calculations.
\end{proof}

\section{Main theorem} \label{sect_main}

Throughout this section we fix $(r, n, m) \in \Omega$ and we put
$\Lambda := \Lambda (r, n, m)$. We summarize our findings in the
theorems describing $\frakZ (\Lambda)$ and $\frakZ' (\Lambda)$.
First we introduce some additional notation.

Let $R$ be a commutative ring graded by $\bbZ$ and $M$ a graded
$R$-module. We put
\[
T (R, M) := \left\{
\begin{bmatrix}
r & m \\ 0 & r
\end{bmatrix}
\mid \text{$r \in R$ and $m \in M$} \right\}.
\]
If we endow $T (R, M)$ with the obvious matrix multiplication,
then it becomes a graded ring, with the grading inherited from
those in $R$ and $M$. Moreover, for $p \in \bbZ$ we define the
graded $R$-module $M [p]$ by $M [p]_q := M [p + q]$ for $q \in
\bbZ$.

We view $\bbF^{\bbN} := \prod_{q \in \bbN} \bbF$ as a graded
$\bbF$-module concentrated in degree $0$. Consequently, if $p \in
\bbN$, then the morphism $\bbF [X^p] \to \bbF$, $f \mapsto f (0)$,
gives $\bbF^{\bbN}$ a structure of a graded $\bbF [X^p]$-module,
where in $\bbF [X^p]$ we have the grading coming from the usual
degree of polynomials.

We have the following description of the graded center of
$\Lambda$.

\begin{theorem}
Let $(r, n, m) \in \Omega$ and $R := \frakZ (\Lambda (r, n, m))$.
Then
\[
R \simeq
\begin{cases}
T (\bbF [X], \bbF^{\bbN}) & \text{$(r, n, m) = (1, 1, 0)$ and
$\chr \bbF = 2$},
\\
T (\bbF [X^2], \bbF^{\bbN}) & \text{$(r, n, m) = (1, 1, 0)$ and
$\chr \bbF \neq 2$},
\\
\bbF [X^n] & \text{$r = n$, $(r, m) \neq (1, 0)$, and $2 \mid n
\cdot \chr \bbF$},
\\
\bbF [X^{2 n}] & \text{$r = n$, $(r, m) \neq (1, 0)$, and $2 \nmid
n \cdot \chr \bbF$},
\\
T (\bbF, \bbF^{\bbN} \oplus \bbF^{\bbN} [-n]) & \text{$(r, n, m) =
(1, 2, 0)$},
\\
T (\bbF, \bbF^{\bbN} [-n]) & \text{$(r, m) \neq (1, 0)$ and $r = n
- 1$},
\\
T (\bbF, \bbF^{\bbN}) & \text{$(r, m) = (1, 0)$ and $r \neq n - 1,
n$},
\\
\bbF & \text{otherwise}.
\end{cases}
\]
In particular, $R_{\red} \neq \bbF$ if and only if $r = n$, and
$R_{\nil} \neq 0$ if and only if either $r = n - 1$ and $r = 1$
and $m = 0$.
\end{theorem}

Let $(r, n, m) \in \Omega$. Then $\gldim \Lambda (r, n, m) =
\infty$ if and only if $r = n$. Moreover, there exists an object
$X$ in $\calD^b (\Lambda (r, n, m))$ such that $\tau X \simeq
\Sigma^p X$ for some $p \in \bbZ$ if and only if either $r = n -
1$ and $r = 1$ and $m = 0$. Consequently, the above theorem
implies the main theorem of the paper.

The following theorem is the analogue of the above one for ``the
commutative version'' of the graded center.

\begin{theorem}
Let $(r, n, m) \in \Omega$ and $R := \frakZ' (\Lambda (r, n, m))$.
Then
\[
R \simeq
\begin{cases}
T (\bbF [X], \bbF^{\bbN}) & \text{$(r, n, m) = (1, 1, 0)$},
\\
\bbF [X^n] & \text{$r = n$ and $(r, m) \neq (1, 0)$},
\\
T (\bbF, \bbF^{\bbN} \oplus \bbF^{\bbN} [-n]) & \text{$(r, n, m) =
(1, 2, 0)$},
\\
T (\bbF, \bbF^{\bbN} [-n]) & \text{$(r, m) \neq (1, 0)$ and $r = n
- 1$},
\\
T (\bbF, \bbF^{\bbN}) & \text{$(r, m) = (1, 0)$ and $r \neq n - 1,
n$},
\\
\bbF & \text{otherwise}.
\end{cases}
\]
\end{theorem}

\bibsection

\end{document}